\documentclass[11pt]{amsart}
\usepackage[utf8]{inputenc}
\usepackage{url, amssymb, enumerate}
\newtheorem{theorem}{Theorem}[section]
\newtheorem{lemma}[theorem]{Lemma}
\newtheorem{proposition}[theorem]{Proposition}
\newtheorem{corollary}[theorem]{Corollary}
\theoremstyle{definition}
\newtheorem{definition}[theorem]{Definition}
\newtheorem{example}[theorem]{Example}
\newtheorem{question}[theorem]{Question}
\theoremstyle{remark}
\newtheorem{remark}[theorem]{Remark}
\numberwithin{equation}{section}
\newcommand{\onto}{\xrightarrow{\textrm{onto}}}
\DeclareMathOperator{\dH}{\Delta}
\DeclareMathOperator{\id}{id}
\DeclareMathOperator{\retr}{\mathcal R}
\DeclareMathOperator{\diam}{diam}
\DeclareMathOperator{\dist}{dist}
\DeclareMathOperator{\sgn}{sign}
\DeclareMathOperator{\Lip}{Lip}

\title{Lipschitz clustering in metric spaces}

\author{Leonid V. Kovalev}
\address{215 Carnegie, Department of Mathematics, Syracuse University, Syracuse, NY 13244, USA}
\email{lvkovale@syr.edu}
\thanks{Supported by the National Science Foundation grant DMS-1764266.}



\begin{document}
\baselineskip6mm
\maketitle

\begin{abstract} In this paper, the Lipschitz clustering property of a metric space refers to the existence of Lipschitz  retractions between its finite subset spaces. Obstructions to this property can be either topological or geometric features of the space. We prove that uniformly disconnected spaces have the Lipschitz clustering property, while for some connected spaces, the lack of sufficiently short connecting curves turns out to be an obstruction. This property is shown to be invariant under quasihomogeneous maps, but not under quasisymmetric ones.   
\end{abstract}

\section{Introduction} 

A metric space $X$ can be viewed as the first member of an infinite sequence of nested metric spaces $(X(n), \Delta)$, $n=1, 2, \dots$ where the elements of $X(n)$ are subsets of $X$ with at most $n$ elements, and $\Delta$ is the Hausdorff metric. The topology and geometry of $X(n)$ can be difficult to grasp  even for simple spaces $X$, e.g.,~\cite{AMS,BorovikovaIbragimov,BorsukUlam,Tuffley}. This paper will show how the relations between the finite subset spaces $X(n)$ reflect the cluster tendency of $X$. 

The detection of clusters in a finite subset of a metric space is an important part of statistical data analysis, and its solution is usually sought in the form of an algorithm. Our approach is different in that we treat the problem as a purely mathematical one, focusing on the existence of a Lipschitz continuous map $\retr\colon X(n)\to X(k)$, $k<n$, such that $\retr$ acts as the identity on $X(k)$. If such $\retr$ exists, the elements of $\retr(A)$ can be interpreted as centers of (at most $k$) clusters of a set $A\in X(n)$. Each point  $x\in A$ can be assigned to a cluster based on which point of $\retr(A)$ is nearest to $x$. 
Among all metric spaces, those that satisfy the ultrametric inequality~\eqref{eq:ultrametric-def} are exceptionally well suited for clustering~\cite{FOUCHAL2013219,MurtaghDownsContreras}. The first of our main results, Theorem~\ref{thm:ultra}, shows that the same holds when cluster tendency of $X$ is quantified by the Lipschitz constants of retractions $\retr\colon X(n)\to X(k)$.   

Section~\ref{sec:linear-sets} concerns the existence of Lipschitz or H\"older continuous retractions for subsets of $\mathbb R$. It simplifies and extends some of the results in~\cite{Kovalev2015}. 
Every  additive subgroup  $G\subset \mathbb R$ supports Lipschitz retractions of $G(n)$ (Corollary~\ref{cor:additive-subgroup}). On the other hand, there exist subsets of $\mathbb R$ with H\"older continuous retractions only (Proposition~\ref{prop:delete-min}). In Corollary~\ref{cor:nofactor} we will see that a Lipschitz retraction $X(n)\to X(k)$ does not always factor through retractions between intermediate finite subset spaces. 

A metric space $X$ that supports Lipschitz retractions $X(n)\to X(k)$ for all $n>k\ge 1$ is said to have the Lipschitz Clustering Property (LCP).
Section~\ref{sec:quasiconvexity} relates the LCP to a better understood class of metric spaces: quasiconvex ones, where any two points $x, y$ can be joined by a curve of  length comparable to the distance between $x$ and $y$. Our second main result, Theorem~\ref{thm:quasiconvex}, shows that an LCP space that contains a bi-Lipschitz image of a line segment must be locally quasiconvex. On the other hand, it is possible for a metric space to have the LCP without containing any rectifiable curves (Example~\ref{ex:snowflake}).   

Section~\ref{sec:invariance} concerns the invariance of the LCP under certain transformations of metric spaces: quasisymmetric and quasihomogeneous maps, products, disjoint unions, etc. The paper concludes with a list of open questions in Section~\ref{sec:questions}. 
 
\section{Definitions and preliminary results}~\label{sec:prelim}

When $A$ is a finite set, $|A|$ denotes its number of elements. Given a metric space $X$ and a positive integer $n$, let $X(n) = \{A\subset X\colon 1\le |A|\le n\}$. The space $X(n)$ is called the $n$th finite subset space of $X$ (other names, such as ``symmetric product'',  appear in the literature). It is a metric space with respect to the \emph{Hausdorff metric}
\begin{equation}\label{def-dh}
\dH(A, B) = \max\left(\sup_{a\in A} \dist(a, B), \sup_{b\in B} \dist(b, A)\right).
\end{equation}
Sometimes we write $\dH_{X}$ instead of $\dH$  to disambiguate the underlying metric space. The notation $\dist $ means the infimal distance $\dist(A, B)=\inf\{d(a, b)\colon a\in A, b\in B\}$.

The natural embeddings $X=X(1)\subset X(2)\subset\cdots$ are isometric with respect to $\dH$, which allows us to consider $X(k)$ as a subset of $X(n)$ when $k < n$.  

If $Y\subset X$, a map $\retr\colon X\to Y$ is a \emph{retraction} if its restriction to $Y$ is the identity map. A set $Y$ for which such $\retr$ exists is a \emph{retract} of $X$. In the context of geometric embeddings of metric spaces, it is desirable for the range of an embedding to be a Lipschitz retract of the ambient space~\cite{LangPlaut}.

A map $f\colon X\to Y$ is called \emph{Lipschitz} if there is a number $L\geq 0$ such that $d_Y(f(x),f(x'))\leq L d_X(x,x')$ for all $x,x'\in X$. We sometimes emphasize the value of $L$ by saying that $f$ is $L$-Lipschitz. The least of such numbers $L$ is denoted by $\Lip(f)$ and is called the \emph{Lipschitz constant} of $f$. A map $f$ is called  $L$-bi-Lipschitz if 
$L^{-1}d_X(x,x') \le d_Y(f(x),f(x'))\le L d_X(x,x') $ for all $x,x'\in X$.

\begin{definition}\label{def:LCP}
A metric space $X$ has the  \emph{Lipschitz Clustering Property (LCP)} if for every $n$ there exists a Lipschitz retraction $\retr_n$ from $X(n)$ onto $X(n-1)$. 
\end{definition}

Definition~\ref{def:LCP} implies, via composition of maps $\retr_n$, the existence of Lipschitz retractions $X(n)\onto X(k)$ for any $n> k\ge 1$.  
The class of LCP spaces includes Euclidean and Hilbert spaces~\cite{Kovalev2016} and, more generally, Hadamard spaces~\cite{BacakKovalev}. It does not include the circle $S^1$~\cite[Proposition 2.2]{AkoforKovalev} or any space that retracts onto a circle.

\begin{definition}\label{def:quasiconvex}
A metric space $(X, d)$ is \emph{quasiconvex} if there exists a constant $C$ such that any two points $x, y\in X$ can be joined by a curve of length at most $C\,d(x, y)$.  
\end{definition}

Although Definitions~\ref{def:LCP} and  \ref{def:quasiconvex} do not look  similar, they have something in common: both properties are inherited by Lipschitz retracts of the space. A stronger connection between them will be established in~\S\ref{sec:quasiconvexity}.

\begin{definition}\label{def:delta-n}
The \emph{minimum separation function} $\delta_n\colon X(n)\to [0, \infty)$ is defined as follows:
$\delta_n(A)=0$ if $|A|<n$, and 
$\delta_n(A) = \min \{d_X(a, b) \colon a, b\in A, \ a\ne b\}$ if $|A|=n$. 
\end{definition}

It is easy to see that $\delta_n$ is a $2$-Lipschitz function which vanishes precisely on $X(n-1)$. Moreover, we have~\cite[Lemma 3.1]{AkoforKovalev}: 
\begin{equation}\label{distances-to-smaller}
\frac12\delta_n(A)\le  \inf\{\dH(A, B)\colon B\in X(n-1)\} \le \delta_n(A).   
\end{equation}

The relevance of $\delta_n$ to Lipschitz clustering is indicated by the following facts.

\begin{lemma} \cite[Lemma 3.2]{AkoforKovalev} \label{lem:displacement}
If $\retr\colon X(n)\to X(n-1)$ is an $L$-Lipschitz retraction, then  for every $A\in X(n)$ we have 
\begin{equation}\label{eq:displacement}
\dH(\retr(A), A)\le (L+1)\delta_n(A).    
\end{equation}
\end{lemma}
  
\begin{lemma} \label{lem:bijection} \cite[Lemma 2.5]{Akofor}
If $A, B\in X(n)$ and $\max(\delta_n(A), \delta_n(B)) > 2\dH(A, B)$, then there exists a bijection $\phi\colon A\to B$ such that $d_X(a, \phi(a))\le \dH(A, B)$ for all $a\in A$.
\end{lemma}

\section{Ultrametrics and uniformly disconnected spaces}\label{sec:ultrametric}

A metric space $(X, d)$ is \emph{ultrametric} if 
\begin{equation}\label{eq:ultrametric-def}
d(x, y)\le \max(d(x, z), d(y, z)) \quad \text{for all } x, y, z\in X.
\end{equation}

A metric space $(X, d)$ is \emph{uniformly disconnected} if there exists a constant $c>0$ such that any finite sequence $x_0, \dots, x_n$ in $X$ satisfies 
\begin{equation}\label{eq:unifdisconnect-def}
\max_{1\le k\le n}d(x_k, x_{k-1}) \ge c\, d(x_0, x_n).
\end{equation}

Uniformly disconnected spaces were introduced by David and Semmes in~\cite[Chapter 15]{DavidSemmes} with a different but equivalent definition; see also Section 14.24 in~\cite{Heinonen}. A metric space $(X, d)$ is uniformly disconnected if and only if there exists an ultrametric $\rho$ on $X$ such that the ratio $d/\rho$ is bounded between two positive constants~\cite[Proposition 15.7]{DavidSemmes}. 

\begin{theorem}\label{thm:ultra} Let $(X, d)$ be an ultrametric space. For any $n > m\ge 1$ and any $L>1$ there exists a retraction $\retr\colon X(n)\to X(m)$ with $\Lip(\retr)\le L$. If $X$ is compact, one can take $L=1$. 
\end{theorem}

The proof of Theorem~\ref{thm:ultra} is based on a proposition that covers a wider class of metric spaces.

\begin{proposition}\label{prop:ultra} Let $(X, d)$ be a metric space. 
Suppose there exist constants $L\ge 1$ and $b\in (0, 1)$ and a family of $L$-Lipschitz maps $\tau_k\colon X\to X$, $k\in \mathbb Z$, such that for every $x, y\in X$ and every $k\in \mathbb Z$ we have: 
\begin{equation}\label{eq:tau-displacement}
    d(\tau_k(x), x)\le Lb^k;
\end{equation}
\begin{equation}\label{eq:tau-separation}
\text{either }\tau_k(x)=\tau_k(y)\text{ or }
d(\tau_k(x), \tau_k(y))\ge L^{-1}b^k.
\end{equation}
Then for any $n > m\ge 1$ there exists a retraction $\retr\colon X(n)\to X(m)$ with $\Lip(\retr)\le 2L^3 b^{-1}+1$. 
\end{proposition}

\begin{proof} Using~\eqref{eq:tau-separation} and the $L$-Lipschitz property of $\tau_k$ we obtain 
\begin{equation}\label{eq:equal-tau2}
d(x, y)< L^{-2}b^k \implies  \tau_k(x)=\tau_k(y). 
\end{equation} 
Hence for any bounded set $A\subset X$ there exists $k\in \mathbb Z$ such that $\tau_k(A)$ consists of a single point.  On the other hand,~\eqref{eq:tau-displacement} implies that $|\tau_k(A)|=|A|$ for all sufficiently large $k$. 

Define $\retr\colon X(n)\to X(m)$ as follows: if $|A|\le m$, then $\retr(A)=A$. Otherwise let $\mu(A) = \max\{k\colon |\tau_k(A)|\le m\}$ and define $\retr(A)=\tau_{\mu(A)}(A)$. 
By definition, $\retr$ is a retraction onto $X(m)$, so it remains to check its Lipschitz property. It is convenient to let $\mu(A) = \infty$ when $A\in X(m)$.   

Given $A, B\in X(n)$, consider three cases. 

\textsc{Case 1}: $\mu(A)=\mu(B)=\infty$. This case is trivial: $\dH(\retr(A), \retr(B)) = \dH(A, B)$. 
 
\textsc{Case 2}: $\mu(A)=\mu(B)<\infty$. Let $k=\mu(A)$ and use the Lipschitz property of $\tau_k$ to obtain 
\[
\dH(\retr(A), \retr(B)) = \dH(\tau_k(A), \tau_k(B))  
\le L \dH(A, B) 
\]

\textsc{Case 3}: $\mu(A) \ne \mu(B)$. We may assume $\mu(A) < \mu(B)$. Let $k=\mu(A)+1$. Since $|\tau_k(A)| > m \ge |\tau_k(B)|$, there exists a point $a_0\in A$ such that $\tau_k(a_0)\notin \tau_k(B)$. By virtue of~\eqref{eq:equal-tau2} we have $d(a_0, b)\ge L^{-2}b^k$ for all $b\in B$. Hence
\begin{equation}\label{eq:B-far-A}
\dH(A, B)\ge L^{-2}b^{\mu(A)+1}.
\end{equation}
The assumption~\eqref{eq:tau-displacement} implies that $\dH(\retr(A), A)$ and $\dH(\retr(B), B)$ are bounded by $Lb^{\mu(A)}$. Therefore, 
\[
\begin{split}
\dH(\retr(A), \retr(B)) 
& \le \dH(\retr(A), A) + \dH(\retr(B), B)  + \dH(A, B) 
\\ & \le 2L b^{\mu(A)} + \dH(A, B) 
\\ & \le (2L^3 b^{-1}+1) \dH(A, B) 
\end{split}
\]
where the last step uses~\eqref{eq:B-far-A}.  
\end{proof}

\begin{proof}[Proof of Theorem~\ref{thm:ultra}]
For $k\in \mathbb Z$ consider the set of all open balls $B(x, 2^{-k})$ with radius $2^{-k}$ and arbitrary center $x\in X$. By the ultrametric inequality~\eqref{eq:ultrametric-def}, for any $x, y\in X$ we have either $B(x, 2^{-k})=B(x, 2^{-k})$ or $\dist(B(x, 2^{-k}), B(y, 2^{-k}))\ge 2^{-k}$. Choose a set of centers $C_k$ such that the balls $B(p, 2^{-k})$, $p\in C_k$, are disjoint and cover $X$. Define $\tau_k\colon X\to X$ so that $\tau_k(x)=p$ when $x\in B(p, 2^{-k})$ with $p\in C_k$. 

By construction, \eqref{eq:tau-displacement} and \eqref{eq:tau-separation} hold with $L=1$ and $b=1/2$. To check the Lipschitz property of $\tau_k$, suppose that $x\in B(p, 2^{-k})$ and $y\in B(q, 2^{-k})$ where $p, q\in C_k$ are distinct. By the ultrametric inequality 
\[
d(\tau_k(x), \tau_k(y)) = 
d(p, q)\le \max(2^{-k}, d(x, y)) = d(x, y)
\] which shows that $\tau_k$ is $1$-Lipschitz. 

Proposition~\ref{prop:ultra} provides a retraction $\retr\colon X(n)\to X(m)$ with $\Lip(\retr)\le 5$. This estimate can be improved with a metric transform as follows. 

Given $L>1$, choose $\alpha>1$ large enough so that $L^\alpha \ge 5$. Since the function $d^\alpha$ satisfies the ultrametric inequality, we can apply the preceding argument to the space $(X, d^\alpha)$ and obtain a retraction $\retr\colon (X(n), \Delta^\alpha)\to (X(m), \Delta^\alpha)$ with 
\[
\dH(\retr(A), \retr(B))^\alpha \le 
5 \dH(A, B)^\alpha \quad \text{for all } A, B\in X(n).
\]
In terms of the original metric $\Delta$, we have $\Lip(\retr)\le 5^{1/\alpha}\le L$ as claimed. 

It remains to consider the case of compact $X$. For each $j\in \mathbb N$ we have a retraction $\retr_j\colon X(n)\to X(m)$ with $\Lip(\retr_j)\le 1 + 1/j$. By the Arzel\`a-Ascoli theorem, a subsequence of $\{\retr_j\}$ converges to a map $\retr\colon X(n)\to X(m)$, which is easily seen to be a $1$-Lipschitz retraction. 
\end{proof}

\begin{corollary}\label{cor:unif-disconnect}
Let $(X, d)$ be an uniformly disconnected metric space. There is a constant $L\ge 1$ such that for any $n > m\ge 1$ there exists an $L$-Lipschitz retraction $\retr\colon X(n)\to X(m)$. 
\end{corollary}

\begin{proof} By~\cite[Proposition 15.7]{DavidSemmes} the space $X$ supports an ultrametric $\rho$ such that $C^{-1}d\le \rho\le d$ for some $C\ge 1$. By Theorem~\ref{thm:ultra} the finite subset spaces of $(X, \rho)$ admit $2$-Lipschitz retractions. In terms of the original metric $d$ these retractions are $L$-Lipschitz with $L=2C^2$. 
\end{proof}

\section{Linear sets}\label{sec:linear-sets}

The availability of total order on $\mathbb R$ allows for a more precise version of Lemma~\ref{lem:bijection}.

\begin{lemma} \label{lem:bijection-line}  
If $A, B\in \mathbb R(n)$ and $\max(\delta_n(A), \delta_n(B)) > 2\dH(A, B)$, then there exists an order-preserving (i.e., increasing) bijection $\phi\colon A\to B$ such that $|a - \phi(a)|\le \dH(A, B)$ for all $a\in A$. Moreover, 
\begin{equation}\label{eq:delete-min-lem}
\dH(A\setminus \{\min A\}, B\setminus \{\min B\}) \le \dH(A, B). 
\end{equation}
\end{lemma}

\begin{proof} Without loss of generality $\delta_n(A) > 2\dH(A, B)$. Let $\phi$ be as in  Lemma~\ref{lem:bijection}. Given $a_1, a_2\in A$ with $a_1<a_2$, let $b_k=\phi(a_k)$ for $k=1, 2$. By the triangle inequality we have
\[
b_2 - b_1 \ge (a_2-a_1) - |a_1-b_1| - |a_2-b_2| 
\ge \delta_n(A) - 2\dH(A, B) > 0.
\]
proving that $\phi$ is order-preserving. Hence the restriction of $\phi$ to $A\setminus \{\min A\}$ is a bijection onto $B\setminus \{\min B\}$, which implies~\eqref{eq:delete-min-lem}. 
\end{proof}

The fact that the real line $\mathbb R$ has the LCP is known~\cite[Lemma 4.2]{Kovalev2015}. But 
Proposition~\ref{prop:line-simple} also gives a simple explicit formula for  Lipschitz retractions on the line, which we will use later.  

\begin{proposition}\label{prop:line-simple} Given a set $A\in \mathbb R(n)$, $n\ge 2$, define $s_A(x)=|A\cap (-\infty, x)|$ for $x\in \mathbb R$. The map $\retr$, defined by 
\begin{equation}\label{eq:line-simple}
\retr(A) = \{x - \delta_n(A) s_A(x) \colon x\in A\}
\end{equation}
is a $(4n-3)$-Lipschitz retraction of $\mathbb R(n)$ onto $\mathbb R(n-1)$. 
\end{proposition}

\begin{proof} If $A\in \mathbb R(n-1)$, then $\delta_n(A)=0$ and therefore $\retr(A)=A$. If $A\in \mathbb R(n)$ has $n$ elements $a_1<\dots<a_n$, then $a_{k+1}-a_{k}=\delta_n(A)$ for some $k\in \{1, \dots, n-1\}$. This implies $a_{k+1} - \delta_n(A) s_A(a_{k+1}) = a_{k} - \delta_n(A) s_A(a_{k})$, hence $\retr(A)\in \mathbb R(n-1)$. Thus, $\retr$ is a retraction onto $\mathbb R(n-1)$. 

By the definition of $s_A$ we have  $s_A(x)\le n-1$ for all $x\in A$, and therefore
\begin{equation}\label{eq:better-displacement}
\dH(\retr(A), A)\le (n-1)\delta_n(A).
\end{equation}
By the triangle inequality, for all $A, B\in \mathbb R(n)$ we have 
\begin{equation}\label{eq:line-1}
\dH(\retr(A), \retr(B))\le \dH(A, B) + (n-1)(\delta_n(A)+\delta_n(B)). 
\end{equation}
If $\delta_n(A)+\delta_n(B)\le 4\dH(A, B)$, then~\eqref{eq:line-1} implies $\dH(\retr(A), \retr(B))\le (4n-3)\dH(A, B)$ as claimed. 

Suppose $\delta_n(A)+\delta_n(B)> 4\dH(A, B)$. By Lemma~\ref{lem:bijection-line} there exists an order-preserving bijection $\phi\colon A\to B$. Note that $s_B(\phi(x))=s_A(x)$ for all $x\in A$. 

Given a point $z\in \retr(A)$, write it as 
$z = x - \delta_n(A) s_A(x)$ for some $x\in A$, and let $w = \phi(x) - \delta_n(B) s_B(\phi(x)) \in \retr(B)$.  Then
\[
|z-w|\le |x-\phi(x)| + |\delta_n(A)-\delta_n(B)| s_A(x) 
\le \dH(A, B) + 2(n-1)\dH(A, B)  
\]
where the last step uses $\Lip(\delta_n)\le 2$. 
As $z$ runs through the points of $\retr(A)$, the corresponding point $w$ runs through all points of $\retr(B)$ because $\phi$ is a bijection. It follows that $\dH(\retr(A), \retr(B))\le (2n-1)\dH(A, B)$ in this case.  
\end{proof}

\begin{corollary}\label{cor:additive-subgroup} Every additive  subgroup $G$ of $\mathbb R$ has the LCP. Furthermore, the intersection of $G$ with any interval has the LCP as well. 
\end{corollary}

\begin{proof} Let $\retr $ be as in~\eqref{eq:line-simple}. 
If $A\in G(n)$, then $\delta_n(A)\in G$, hence $\retr(A)\subset G$. Moreover, $\min \retr(A)=\min A$ and $\max \retr(A) \le \max A$, which means that the subsets of any interval do not move out of the interval under $\retr$. 
\end{proof}

\begin{remark}\label{rem:symmetric} 
The retraction~\eqref{eq:line-simple} moves the points of $A$ toward $\min A$. One could replace  $s_A$ with the sign-counting function 
$
\sigma_A(x) = \frac12 \sum_{y\in A} \sgn(y-x)
$
which results in a retraction that moves the points of $A$ toward the median of $A$. However, this map does not preserve additive subgroups because $\sigma_A$ is not integer-valued. 
\end{remark}
 
So far we saw that the LCP holds both for connected subsets of $\mathbb R$ and for uniformly disconnected ones. The following result shows that it also holds for compact subsets with finitely many components.  

\begin{theorem}\label{thm:finite-union}
Suppose that $X\subset \mathbb R$ is a union of disjoint compact intervals $I_k$, $k=1, \dots, m$, some of which may degenerate into points. Then $X$ has the LCP. 
\end{theorem}

\begin{proof} Let $M = \max_{1\le k\le m} \diam I_k$.
Fix $n\ge 2$ and let $\retr \colon \mathbb R(n)\to \mathbb R(n-1)$ be as in Proposition~\ref{prop:line-simple}. 

After applying a suitable bi-Lipschitz transformation   $F\colon \mathbb R\to \mathbb R$, we can achieve $\dist(I_k, I_j)\ge 3 n M$  whenever $j\ne k$. To be specific, $F$ could be a piecewise-linear function that does not change the diameter of any interval $I_k$ but increases the distances between them. 

Let $\widetilde{X}=\{x\in \mathbb R\colon \dist(x, X)\le nM\}$. Each gap between the components of $\widetilde{X}$ is at least a third of the corresponding gap between the components of $X$. Therefore,  the nearest-point projection $\rho\colon\widetilde{X}\to X $, which sends each point of $\widetilde{X}$ to the nearest point of $X$, is $3$-Lipschitz. 

We partition $X(n)$ as $X(n)=U\cup V$ where $U$ consists of all $A\in X(n)\setminus X(n-1)$ such that  $|A\cap I_k|\le 1$ for all $k=1, \dots, m$, and $V=X(n)\setminus U$. In other words, we have $A\in U$ if and only if $A$ intersects $n$ of the intervals $I_1, \dots, I_m$. The set $U$ is empty when $n > m$. 

Define $\retr_X\colon X(n)\to X(n-1)$ as follows.
\begin{equation}\label{eq:cases-retr}
\retr_X(A) = \begin{cases}
A\setminus \{\min A\} & \text{if } A\in U \\ 
\rho(\retr(A)) & \text{if } A\in V 
\end{cases}
\end{equation}

Since $\rho$ is only defined on $\widetilde{X}$, we must show that $\retr(A)\subset \widetilde{X}$ for $A\in V$. Such sets have $\delta_n(A)\le M$ since either $|A|<n$ or two of the points of $A$ lie in the same component of $X$. From~\eqref{eq:better-displacement} we have  $\dH(\retr(A), A)\le  (n-1)M$, hence $\retr(A)\subset \widetilde{X}$. The Lipschitz continuity of $\rho$ and $\retr$ implies that $\retr_X$ is Lipschitz on $V$. 

The set $U$ consists of $\binom{m}{n}$ connected components, based on which $n$ of the intervals $I_1, \dots, I_m$ contain a point of $A$. The distance (in the metric $\dH$) between any two  components of $U$ is at least $3nM$, because of the gaps between the intervals $I_k$. For the same reason, $\dH(A, B)\ge 3nM$ whenever $A\in U$ and $B\in V$. 

Let $A$ and $B$ be two sets in the same connected component of $U$, which means they intersect the same collection of intervals $I_k$. Since the length of each interval is bounded by $M$, we have  $\dH(A, B)\le M$. Furthermore, the gaps between the intervals force 
$\delta_n(A)\ge 3nM$. Inequality~\eqref{eq:delete-min-lem} shows that $\retr_X$ is $1$-Lipschitz on each connected component of $U$. This completes the proof.  
\end{proof}

The map $A\mapsto A\setminus \{\min A\}$ merits  further consideration. It is not continuous on $\mathbb R(n)$ when $n\ge 3$, as the example of sets $\{0, \epsilon, 1\}$ and $\{0, 1, 1+\epsilon\}$ shows: the Hausdorff distance between these sets is $\epsilon$, but it increases to $1-\epsilon$ when the minimal elements are removed.  However, it provides continuous retractions of $X(n)$ for some linear sets $X$ to which Theorem~\ref{thm:finite-union} does not apply. One such example is given below.  

\begin{proposition} \label{prop:delete-min} Let $X=\{0\}\cup\{k^{-1}\colon k\in \mathbb N\}$. Fix $n\ge 2$. For $A\in X(n)$ let $\retr(A)=A\setminus \{\min A\}$ if $|A|=n$ and $\retr(A)=A$ if $|A|<n$.  Then: 
\begin{enumerate}[(a)] 
\item $\retr$ is a H\"older continuous retraction of $X(n)$ onto $X(n-1)$, with H\"older exponent $1/2$;
\item $X$ does not have the LCP. 
\end{enumerate} 
\end{proposition}

\begin{proof} 
(a) Consider a set $A\in X(n)$ of the form $A=\{a_1, \dots, a_n\}$ where $a_1<a_2<\dots < a_n$. We have $\dH(\retr(A), A)=a_2-a_1$. Also, 
\[
\delta_n(A) = \min\{a_{k+1}-a_k \colon k=1, \dots, n-1\} 
\ge \min(a_2-a_1, a_2^2)
\] 
because $|a-b|\ge ab$ for any two points $a, b\in X$. Thus
\begin{equation}\label{eq:delmin-0}
\dH(\retr(A), A)\le \sqrt{\delta_n(A)}.
\end{equation}
Recalling~\eqref{distances-to-smaller} we conclude that 
\begin{equation*}
\dH(\retr(A), \retr(B)) \le \dH(A, B) + \sqrt{2\dH(A, B)} \le 3\sqrt{\dH(A, B)} \quad  \text{if } |A|=n \text{ and } |B|<n.
\end{equation*}
 
It remains to consider the case of two sets $A, B\in X(n)\setminus X(n-1)$. If $\max(\delta_n(A), \delta_n(B))> 2\dH(A, B)$, then 
$\dH(\retr(A), \retr(B))\le \dH(A, B)$ by  Lemma~\ref{lem:bijection-line}. Assume 
$\max(\delta_n(A), \delta_n(B))\le 2\dH(A, B)$. By~\eqref{eq:delmin-0} we have 
\[\begin{split}
\dH(\retr(A), \retr(B)) & \le \dH(A, B) + \sqrt{\delta_n(A)} + \sqrt{\delta_n(B)} \\ & \le \dH(A, B) + 2\sqrt{2\dH(A, B)} 
\le 4 \sqrt{\dH(A, B)}. 
\end{split} \]
This completes the proof of~(a).

(b) If $t\in X\setminus \{0, 1\}$, then the neighbors of $t$ in $X$ are $t/(1+t)$ and $t/(1-t)$. The distances from $t$ to its neighbors are $t^2/(1+t)$ and  $t^2/(1-t)$, respectively. Hence
\begin{equation}\label{eq:nolip-0}
\dH(A, B)\ge t^2/2 \quad \text{if } t\in A\vartriangle B,     
\end{equation}
where $A, B\in X(n)$ and $A\vartriangle B = (A\setminus B)\cup (B\setminus A)$.

Suppose that $\retr\colon X(4)\to X(3)$ is an $L$-Lipschitz retraction. Choose $x, y, z\in X$ such that 
\begin{equation}\label{eq:nolip-1}
0<2x<y<z, \quad 2Lx^2 < y^2, 
\quad \text{and} \quad  2(L+1)(z-y) < x.
\end{equation}
For example, one can take $x=1/k^3$, $y=1/(k^2+1)$, and $z=1/k^2$ where $k\in \mathbb N$ is large enough that~\eqref{eq:nolip-1}  is satisfied. 

Let $A=\{0, y, z\}$ and $B=\{0, x, y, z\}$. Since the distance between consecutive elements of $X\cap [0, x]$ is less than $x^2$, there is a finite sequence of sets $A_j\in X(4)$ that begins with $A_1=A$ and ends with $A_J=B$, such that $\dH(A_j, A_{j+1})\le x^2$ for all $j=1, \dots, J-1$. It follows that $\dH(\retr(A_j), \retr(A_{j+1}))\le Lx^2<y^2/2$ for all $j$. Since $\retr(A_1)=A$ contains $y$, inequality~\eqref{eq:nolip-0} implies $y\in \retr(A_j)$ for all $j$. The same argument shows that $z\in \retr(A_j)$ for all $j$. 
Thus, we have $\retr(B)=\{u, y, z\}$ for some $u\in X$. 

The set $\retr(B)\cap [0, 2x]$ has at most one point because $z>y>2x$. Since $B$ contains both $0$ and $x$, it follows that $\dH( \retr(B), B) \ge x/2$. On the other hand, Lemma~\ref{lem:displacement} yields
\[
\dH( \retr(B), B) \le (L+1)\delta_4(B) \le 
(L+1)(z-y) < x/2 
\]
where the last step uses~\eqref{eq:nolip-1}. This contradiction completes the proof. 
\end{proof}  
 
The proof of Proposition~\ref{prop:delete-min} shows the non-existence of a Lipschitz retraction from $X(4)$ to $X(3)$. In contrast, for every subset $X\subset \mathbb R$ and any $n\ge 2$ one has $1$-Lipschitz retractions $A\mapsto \{\max(A)\}$ from $X(n)$ to $X(1)$, and $A\mapsto \{\min(A), \max(A)\}$ from $X(n)$ to $X(2)$. The exceptional nature of retractions onto $X(1)$ and $X(2)$ was also observed in~\cite{CAOZRQ} in the context of finite subsets of trees. 

\begin{corollary}\label{cor:nofactor} A Lipschitz retraction $X(n)\to X(k)$ does not necessarily factor into a  chain of retractions
\[X(n)\to X(n-1)\to \cdots \to X(k+1)\to X(k).\]
\end{corollary}
\begin{proof}
Let $X$ be the set in Proposition~\ref{prop:delete-min}. The map $X(4)\to X(2)$ sending every set $A$ to $\{\min A, \max A\}$ does not factor through a retraction onto $X(3)$, since none exist.
\end{proof}

Proposition~\ref{prop:delete-min} demonstrates that H\"older continuous clustering may be possible in some settings where Lipschitz clustering is unavailable. 
As another possible instance of this phenomenon, Akofor~\cite{Akofor} proved that for every normed space $X$ there are locally H\"older retractions $X(n)\to X(n-1)$, while the existence of Lipschitz retractions remains unknown~\cite[Question 3.4]{Kovalev2016}.

\section{Quasiconvexity of LCP spaces}\label{sec:quasiconvexity} 

The main result of this section gives a strong necessary condition for the Lipschitz Clustering Property. It exhibits a dichotomy for LCP spaces: they are either well connected by rectifiable curves, or do not admit any bi-Lipschitz maps from an interval. This result  will be used to show that the LCP is not inherited by products or disjoint unions.  

\begin{theorem}\label{thm:quasiconvex} 
Suppose that a metric space $(X, d)$ supports an $L$-Lipschitz retraction $\retr\colon X(4)\to X(3)$ and  the interval $[0, 1]$ admits an $L$-bi-Lipschitz embedding into $X$. Then there exist positive constants $r$ and $M$, depending only on $L$, such that any two points $p, q\in X$ with $d(p, q)\le r$ can be connected by a curve of length at most $Md(p, q)$.
\end{theorem}

The proof of Theorem~\ref{thm:quasiconvex} is preceded by several lemmas. 

\begin{lemma} \cite[Lemma 4.3]{Kovalev2015} \label{lem2015}
Let $Z$ and $X$ be metric spaces with $D:=\diam Z < \infty$. Suppose that
$f \colon Z \to X(n)$ is an $L$-Lipschitz function such that
\begin{equation}\label{eq:lem2015}
\diam f(z_0) > 3(n-1)LD  \quad \text{for some } z_0\in Z.
\end{equation}
Then there are $L$-Lipschitz functions $g, h\colon  Z \to X(n-1)$ such that $f(z) = g(z)\cup h(z)$ for all $z \in Z$. 
Specifically, one can let
\begin{equation}\label{eq:lem2015gh}
\begin{split}
g(z) &= \{x \in f(z) \colon \dist(x, E) \le LD\}; \\
h(z) &= \{x \in f(z) \colon \dist(x, f(z_0)\setminus E) \le LD\} = f(z)\setminus g(z) 
\end{split}
\end{equation}
where $E$ can be any subset of $f(z_0)$ that satisfies $\diam E \le 3LD(|E| - 1)$ and is a maximal such subset with respect to containment.  
\end{lemma}

Lemma~\ref{lem2015} can be refined when the domain $Z$ is an interval and the cardinality of $f(z)$ does not depend on $z$. 

\begin{lemma}\label{lem:constant-card}
Let $X$ be a metric space. 
Let $I\subset \mathbb R$ be an interval, possibly unbounded. Any $L$-Lipschitz map $f \colon I\to X(n)\setminus X(n-1)$ can be decomposed as 
\begin{equation}\label{eq:constant-card}
 f(t)=\{f_1(t), \dots, f_n(t)\}   
\end{equation}
where each function $f_k\colon I\to X$ is $L$-Lipschitz.
\end{lemma}  

\begin{proof} The minimal separation function $\delta_n(f(t))$ is continuous and positive on $I$. Therefore, for every compact subinterval $J$ there exists $\delta>0$ such that $\delta_n(f(t))\ge \delta$ for all $t\in J$. 
Let $\epsilon = \delta/(3L(n-1))$. On any subinterval $S\subset J$ with $\diam S < \epsilon$  one can apply  Lemma~\ref{lem2015} repeatedly to obtain a decomposition of the form~\eqref{eq:constant-card}, because any subset of $f(t)$ with more than one element has diameter at least $\delta$. 

Let $J=\bigcup_{i=1}^N [t_{i-1}, t_i]$ be a partition of $J$ into subintervals of length less than $\epsilon$. We have an $L$-Lipschitz decomposition for each $i$,
\[
f(t)=\{f_1^i(t), \dots, f_n^i(t)\}, \quad t\in  [t_{i-1}, t_i].    
\]
Since $\{f_1^i(t_i), \dots, f_n^i(t_i)\} = \{f_1^{i+1}(t_i), \dots, f_n^{i+1}(t_i)\}$, we can relabel the functions to achieve $f_k^i(t_i)=f_k^{i+1}(t_i)$ for all $k=1, \dots, n$ and all $i=1, \dots N-1$. This produces an $L$-Lipschitz decomposition of $f$ on $J$.

Since $I$ can be partitioned into countably many compact subintervals, the above process of concatenation produces an $L$-Lipschitz decomposition of $f$ on all of $I$.
\end{proof}

\begin{lemma}\label{lem:2become1} If a set $\{p, q\}\subset X$ is connected to some singleton $\{c\}$ by a curve of length $\ell$ in $X(2)$, then the points $p$ and $q$ are connected in $X$ by a curve of length at most $2\ell$.
\end{lemma}

\begin{proof} We may assume $p\ne q$. 
Let $\Gamma\colon [0, \ell]\to X(2)$ be a curve with $\Gamma(0)=\{p, q\}$ and $\Gamma(\ell)=\{c\}$, parametrized by arclength. Since $\delta_2\colon X(2)\to [0,\infty)$ is a continuous function, the set $S = \{t\in [0, \ell] \colon \delta_2(\Gamma(t))=0\}$ is closed. Let $t_0=\inf S$. Since the restriction of $\Gamma$ to $[0, t_0)$ takes values at $X(2)\setminus X(1)$, Lemma~\ref{lem:constant-card} provides a decomposition $\Gamma(t)=\{\gamma_1(t), \gamma_2(t)\}$ where both $\gamma_1$ and $\gamma_2$ are $1$-Lipschitz on $[0, t_0)$. The curves $\gamma_1$ and $\gamma_2$ have the same limit as $t\to t_0-$ because $|\Gamma(t_0)|=1$. Thus, their concatenation is a curve from $p$ to $q$ of length at most $2t_0\le 2\ell$. 
\end{proof}

\begin{proof}[Proof of Theorem~\ref{thm:quasiconvex}] 
By assumption there exists an $L$-bi-Lipschitz embedding $\Gamma\colon [0, 1]\to X$. Let 
\begin{equation}\label{eq:choose-r}
r = \frac{1}{48(L+1)^5}.     
\end{equation}

Fix  distinct points $p,q$ with $d(p, q)\le r$ and let $E=\{p, q\}$. By the triangle inequality there exists $t_0\in \{0, 1\}$ such that $d(\Gamma(t_0), p)\ge 1/(2L)$. 
We may assume $t_0=0$ (otherwise, reverse the parametrization of $\Gamma$). 
Let $I = [0, D]$ where $D=1/(24L^3)$. Observe that  
\begin{equation}\label{eq:diamG}
 \frac{1}{24L^4} \le \diam \Gamma(I)\le \frac{1}{24L^2}
\end{equation}
and therefore 
\begin{equation}\label{eq:distG}
\dist(\Gamma(I), E)\ge \frac{1}{2L} -  \frac{1}{24L^2} 
\ge \frac{11}{24 L}.
\end{equation}

Define $f\colon I\to X(3)$ by $f(t) = \retr(\{\Gamma(t), \Gamma(0), p, q\})$. By construction, $f$ is $L^2$-Lipschitz. 

Let us check that the assumptions of Lemma~\ref{lem2015} are satisfied for the map $f$ with domain $Z=I$ of diameter $D=1/(24L^3)$, distinguished point $z_0 = 0$, Lipschitz constant $L^2$, and with the choice $E=\{p, q\}\subset f(0)$. Indeed, 
\[
\diam f(0) = \diam \{\Gamma(0), p, q\} \ge \frac{1}{2L}
= 12 L^2 D
\]
while 
\[
\diam E = d(p, q) \le r < \frac{1}{8L} = 3L^2D. 
\]
The maximality assumption of Lemma~\ref{lem2015} holds because $E$ is a maximal proper subset of $f(0)$. 

Lemma~\ref{lem2015} provides $L^2$-Lipschitz maps $g, h\colon I\to X(2)$ such that $f(t)=g(t)\cup h(t)$ for all $t\in I$, and, moreover,
\begin{equation}\label{eq:define-g}
g(t) = \{x\in f(t)\colon \dist(x, E)\le L^2 D\}.
\end{equation}
In particular, $g(0)=E=\{p, q\}$ because $f(0) = \{\Gamma(0), p, q\}$.

\textbf{Claim}: There exists $T\in [0, D]$ such that 
\begin{equation}\label{eq:DTclaim}
|g(T)|=1 \quad \text{and} \quad T\le 2(L+1)^2  d(p, q). 
\end{equation}

Assuming this claim for now, let us show how it implies the statement of Theorem~\ref{thm:quasiconvex}. Since 
$g$ is $L^2$-Lipschitz, the set $g(0)=\{p, q\}$ is connected to the singleton $g(T)$ by a curve of length at most $L^2T$ in $X(2)$. By Lemma~\ref{lem:2become1} there is a curve of length at most $2L^2T$ connecting $p$ to $q$ in $X$. By virtue of~\eqref{eq:DTclaim} the statement of Theorem~\ref{thm:quasiconvex} holds with $M = 4L^2(L+1)^2$. 

It remains to prove~\eqref{eq:DTclaim}. Fix $t\in [0, D]$ such that $g(t)$ contains more than one point. Since $g(t)$ and $h(t)$ are disjoint and their union has at most three elements, it follows that $h(t)$ is a singleton. 

Consider the set $A:=\{\Gamma(t), \Gamma(0), p, q\}$. 
By Lemma~\ref{lem:displacement}, the Hausdorff distance between $A$ and $f(t)=\retr(A)$ can be estimated as
\begin{equation}\label{eq:upper-1}
\dH(f(t), A) \le (L+1)\delta_4(A)\le (L+1)d(p, q).  
\end{equation}
But we can also estimate $\dH(f(t), A)$ from below. Indeed,~\eqref{eq:distG} and~\eqref{eq:define-g} imply that 
\begin{equation}\label{eq:lower-1}
\dist(\Gamma(t), g(t))\ge 
\dist(\Gamma(t), E) - L^2D \ge \frac{11}{24L} - \frac{1}{24L} = \frac{5}{12L}.
\end{equation}
The fact that $h(t)$ is a singleton implies that for some $s\in \{0, t\}$ 
\begin{equation}\label{eq:lower-2}
\dist(\Gamma(s), h(t)) \ge \frac{1}{2}d(\Gamma(0), \Gamma(t)) \ge \frac{t}{2L}. 
\end{equation}
Since $A$ contains both $\Gamma(0)$ and $\Gamma(t)$, it follows from~\eqref{eq:lower-1} and \eqref{eq:lower-2} that 
\[
\dH(f(t), A) \ge \min\left( \frac{5}{12L}, \frac{t}{2L}\right). 
\]
Recalling~\eqref{eq:upper-1}, we obtain 
\begin{equation}\label{eq:upper-lower}
\min\left( \frac{5}{12L}, \frac{t}{2L}\right) \le 
(L+1)d(p, q).   
\end{equation}
The right hand side of~\eqref{eq:upper-lower} does not exceed $(L+1)r = 1/(48(L+1)^4)$ which is less than $5/(12L)$ on the left hand side. Hence 
\begin{equation}\label{eq:upper-lower-2}
\frac{t}{2L} \le (L+1)d(p, q).  
\end{equation}

By~\eqref{eq:upper-lower-2}, we have $|g(T)|=1$ for any $T$ such that  $2L(L+1)d(p, q)<T \le D$. To see that this interval is nonempty, recall~\eqref{eq:choose-r} which implies
\[
2L(L+1)d(p, q)\le \frac{2L(L+1)}{48(L+1)^5}
< \frac{1}{24L^3} = D.
\]
This completes the proof of Claim~\eqref{eq:DTclaim} and of the theorem.
\end{proof}

\begin{example}\label{ex:small-gap} For every $\epsilon>0$ the set $X=[-1, 0]\cup [\epsilon, 1]$ has the LCP by Theorem~\ref{thm:finite-union}. However, the Lipschitz constant of any retraction of $X(4)$ onto $X(3)$ cannot be bounded by a constant $L$ independent of $\epsilon$. Indeed, if this was possible, then by  Theorem~\ref{thm:quasiconvex} we would have $r>0$, independent of $\epsilon$, such that any two points  $p, q\in X$ with $|p-q|\le r$ can be connected by a curve in $X$. But this is false when $\epsilon \le r$. 
\end{example}

Under the stronger hypothesis of containing bi-Lipschitz images of \emph{long} line segments, the conclusion of Theorem~\ref{thm:quasiconvex} can be strengthened to global quasiconvexity. 

\begin{corollary}\label{cor:global-qc} 
Suppose that a metric space $X$ supports an $L$-Lipschitz retraction $\retr\colon X(4)\to X(3)$ and  for every $T>0$ the interval $[0, T]$ admits an $L$-bi-Lipschitz embedding into $X$. Then $X$ is quasiconvex. 
\end{corollary}

\begin{proof} Let $X_\epsilon$ be the rescaling of metric space $(X, d)$ by the factor $\epsilon>0$; that is, the metric on $X_\epsilon$ is $\epsilon d$. The retraction $\retr\colon X(4)\to X(3)$ induces a retraction $X_\epsilon(4)\to X_\epsilon(3)$ with the same Lipschitz constant. Also, an $L$-bi-Lipschitz embedding $\Gamma\colon [0, \epsilon^{-1}]\to  X$ induces an $L$-bi-Lipschitz embedding of $[0, 1]$ into $X_\epsilon$, namely $\Gamma_\epsilon(t)=\Gamma(\epsilon^{-1} t)$. 

Applying Theorem~\ref{thm:quasiconvex} to $X_\epsilon$ we find that there exist $M$ and $r$, which depend only on $L$, such that any two points $p, q\in X_\epsilon$ with $\epsilon d(p, q) \le r$ can be joined by a curve $\gamma$ of length at most $M \epsilon d(p, q)$ in $X_\epsilon$. 

Given any two points $p, q\in X$, we can choose $\epsilon>0$ such that $\epsilon d(p, q)\le r$. The previous paragraph provides a curve connecting $p$ to $q$ in $X$, the length of which is at most $M d(p, q)$ in the metric of $X$. 
\end{proof}

\begin{example}\label{ex:parallel-lines}
Let $X=\mathbb R\times \mathbb Z$, considered as a subset of $\mathbb R^2$ with the restriction metric. Since $X$ contains lines but is not a connected space, by Corollary~\ref{cor:global-qc} it does not have the LCP. Thus, additive subgroups of $\mathbb R^2$ do not have the LCP in general, in contrast to Corollary~\ref{cor:additive-subgroup}. 
\end{example}

\begin{example}\label{ex:parabola} Consider the parabola $P=\{(x, y)\in \mathbb R^2\colon y = x^2\}$ with the restriction metric inherited from $\mathbb R^2$. The set $P$ contains $2$-bi-Lipschitz images of arbitrarily long line segments. Since $P$ is not quasiconvex, by Corollary~\ref{cor:global-qc} it does not have the LCP.
\end{example}

\section{Transformations of metric spaces}\label{sec:invariance} 

This section concerns the invariance of the Lipschitz Clustering Property under certain transformations of metric spaces. Since any bi-Lipschitz map $f\colon X\onto Y$ induces a bi-Lipschitz map of $X(n)$ onto $Y(n)$, it follows that LCP is bi-Lipschitz invariant. The following lemma extends this observation. 

\begin{lemma}\label{lem:AKretract}  \cite[Lemma 3.3]{AkoforKovalev} Suppose that $X$ and $Y$ are metric spaces and there exist Lipschitz maps $f\colon X\to Y$ and $g\colon Y\to X$ with $f\circ g=\id_Y$. If $X$ has the LCP, then so does $Y$. 
\end{lemma}

For example, Lemma~\ref{lem:AKretract} applies when $Y\subset X$ and $g\colon Y\to X$ is the inclusion map. In this case, the existence of a Lipschitz  map $f$ with $f\circ g=\id_Y$ means precisely that $Y$ is a Lipschitz retract of $X$. 

The Lipschitz retracts of $\mathbb R^d$ have a transparent characterization when $d=2$~\cite[Theorem~2.11]{HeinonenLip} and a less transparent one for $d>2$ (\cite[Theorem 3.4]{Hohti} and \cite[Theorem~2.12]{HeinonenLip}). 
All these retracts have the LCP by Lemma~\ref{lem:AKretract}. Example~\ref{ex:snowflake} will show that the converse is not true even for connected sets: a connected LCP subset of $\mathbb R^d$ need not be a Lipschitz retract of $\mathbb R^d$. 

The class of quasisymmetric maps in metric spaces~\cite[Chapters 10--12]{Heinonen} contains bi-Lipschitz class as a proper subset. 

\begin{definition}\label{def:qs-map}  A homeomorphism $f\colon X\onto Y$ is \emph{quasisymmetric} if there exists a homeomorphism $\eta \colon [0, \infty) \to [0, \infty)$ such that for any three distinct points $x, u, v$ in $X$ we have
\[
\frac{d_Y(f(x), f(u))}{d_Y(f(x), f(v))}
\le \eta\left( \frac{d_X(x, u)}{d_X(x, v)}\right). 
\]
\end{definition}

Using Theorem~\ref{thm:quasiconvex} we can show that the Lipschitz Clustering Property is not invariant under quasisymmetric maps. 
Indeed, a quasisymmetric image of $\mathbb R$ may be a curve $\Gamma$ that contains both a line segment and an unrectifiable arc such as the von Koch snowflake. This follows from Ahlfors' characterization of such images in terms of the three-point ``bounded turning'' condition~\cite[\S IV.D]{Ahlfors}. 
By Theorem~\ref{thm:quasiconvex}, $\Gamma$ does not have the LCP. 

The following lemma shows the LCP is preserved by certain transformations of the metric $d$, such as the snowflake transform $d\mapsto d^\alpha$, $0<\alpha<1$.  Note that the identity map from $(X, d)$ onto $(X, d^\alpha)$ is quasisymmetric but not bi-Lipschitz. 

\begin{lemma}\label{lem:metric-transform}
Suppose $(X, d)$ is a metric space with the LCP, and $\varphi\colon [0, \infty)\to [0, \infty)$ is a nondecreasing function such that $\varphi\circ d$ is a metric on $X$. If, in addition, there exists a constant $M$ such that 
\begin{equation}\label{eq:doubling-phi}
\varphi(2t)\le M\varphi(t)\quad \text{for all }  t\in [0, \infty), 
\end{equation}
then $(X, \varphi\circ d)$ has the LCP. 
\end{lemma} 

\begin{proof} Consider a retraction $\retr\colon X(n) \to X(n-1)$ that is $L$-Lipschitz with respect to the Hausdorff metric $\dH$ based on $d$. The doubling property~\eqref{eq:doubling-phi} implies that there is a constant $L'$ such that $\varphi(Lt)\le L'\varphi(t)$ for all $t\in [0, \infty)$. Therefore, for any $A, B\in X(n)$ we have 
\[
\varphi(\dH(\retr(A), \retr(B))) 
\le \varphi(L \dH(A, B)) \le L'\varphi(\dH(A, B)) 
\]
which proves the claim. 
\end{proof}
 
\begin{example}\label{ex:snowflake} 
Fix $\alpha\in (1/2, 1)$. By \cite[Proposition 4.4]{Assouad} there exists a map $f\colon [0, 1]\to \mathbb R^2$ such that 
\[
C^{-1} |x-y|^\alpha\le |f(x)-f(y)|\le C|x-y|^\alpha ,\quad x, y\in \mathbb [0, 1]
\]
for some constant $C$.  Lemma~\ref{lem:metric-transform} shows that the snowflake-type curve $\Gamma=f([0, 1])$ has the LCP. On the other hand, $\Gamma$ contains no rectifiable curves and therefore is not a  quasiconvex set. Consequently, it is not a Lipschitz retract of $\mathbb R^2$.
\end{example} 
 
Since Lipschitz maps generally behave well with respect to Cartesian products, one may expect the Lipschitz Clustering Property to be inherited by such products. However, \emph{Rickman's rug}, one of standard examples of a fractal surface (\cite{DiMarco}, \cite{Freeman13}, or \cite[p. 65]{MackayTyson}), provides a counterexample. By definition, a Rickman's rug is the Cartesian product of a line segment $I$ with the ``snowflake'' $I_\alpha$ which is the set $I$ equipped with the metric $|x-y|^\alpha$, $0<\alpha<1$. Both factors $I$ and $I_\alpha$  have the LCP by Theorem~\ref{thm:finite-union} and Lemma~\ref{lem:metric-transform}. However, Theorem~\ref{thm:quasiconvex} will show that the product $I\times  I_\alpha$ does not have the LCP, since it contains some line segments without being locally connected by rectifiable curves. 

The argument from the previous paragraph applies to the disjoint union $X=I\sqcup I_\alpha$ as well. It follows that the disjoint union of two  compact LCP spaces need not have the LCP.

The Lipschitz Clustering Property is preserved by quasihomogeneous maps, which form an intermediate class between bi-Lipschitz and quasisymmetric maps. 

\begin{definition}\label{def:qh-map} \cite{Freeman11,GhamsariHerron,HerronMayer}   A homeomorphism $f\colon X\onto Y$ is \emph{quasihomogeneous} if there exists a homeomorphism $\eta \colon [0, \infty) \to [0, \infty)$ such that for any four distinct points $x_1, \dots, x_4$ in $X$ we have
\begin{equation}\label{eq:qh-map}
\frac{d_Y(y_1, y_2)}{d_Y(y_3, y_4)}
\le \eta\left( \frac{d_X(x_1, x_2)}{d_X(x_3, x_4)}\right) 
\end{equation}
where $y_k=f(x_k)$, $k=1, \dots, 4$. 
\end{definition}

\begin{proposition}\label{prop:qhmaps}
If $f\colon X\onto Y$ is a quasihomogeneous map and $X$ has the LCP, then $Y$ has the LCP as well. \end{proposition}

\begin{proof} Property~\eqref{eq:qh-map} can be equivalently stated as follows: if $t > 0$ and $d_X(x_1, x_2)\le d_X(x_3, x_4)$, then $d_Y(y_1, y_2) \le \eta(t) d_Y(y_3, y_4)$. 

Our first goal is to prove that if 
$A_1, \dots, A_4\in X(n)$ satisfy $\dH_X(A_1, A_2) \le t \dH_X(A_3, A_4)$, then the sets $B_k=f(A_k)$ satisfy 
\begin{equation}\label{eq:induced-qh}
\dH_Y(B_1, B_2)\le \eta(t) \dH_Y(B_3, B_4).
\end{equation}
That is,  the map $X(n)\to Y(n)$ defined by $A\mapsto f(A)$ is $\eta$-quasihomogeneous. A similar observation was made in~\cite[Theorem 3.4]{KovalevTyson} but it was not quantifies as in~\eqref{eq:induced-qh}. 

Because $B_1$ and $B_2$ are interchangeable, to
prove~\eqref{eq:induced-qh} it suffices to show that $\forall y_1\in B_1 \ \exists y_2\in B_2$ such that 
\begin{equation}\label{eq:qh01}
d_Y(y_1, y_2)\le \eta(t) \dH_Y(B_3, B_4).    
\end{equation}
Given $y_1\in B_1$, let $x_1=f^{-1}(y_1)$ and choose $x_2\in A_2$ so that $d_X(x_1, x_2)\le \dH_X(A_1, A_2)$. We claim that the point $y_2:=f(x_2)$ satisfies~\eqref{eq:qh01}. 

After exchanging the roles of $A_3$ and $A_4$ if necessary, we can find $x_3\in A_3$ such that $d_X(x_3, x_4)\ge \dH_X(A_3, A_4)$ for all $x_4\in A_4$. Thus, with the above choice of $x_1, x_2, x_3$ we have 
\[
d_X(x_1, x_2)\le t d_X(x_3, x_4), \quad \forall x_4\in A_4.
\]
The quasihomogeneity of $f$ implies 
\[
d_Y(y_1, y_2)\le  \eta(t) d_Y(y_3, y_4), \quad \forall y_4\in B_4
\]
where $y_3=f(x_3)$. Taking the minimum over $y_4\in B_4$ we obtain~\eqref{eq:qh01}. This completes the proof of~\eqref{eq:induced-qh}. 

By assumption, there exists a Lipschitz retraction $\retr\colon X(n)\to X(n-1)$. For any sets $A, A'\in X(n)$ we have  
$\dH_X(\retr(A), \retr(A')) \le L \dH_X(A, A')$ where $L=\Lip(\retr)$. Inequality~\eqref{eq:induced-qh} yields  
\[
\dH_Y(f(\retr(A)), f(\retr(A')))\le \eta(L) \dH_Y(f(A), f(A')).
\] 
Therefore, the map $\widetilde{\retr}\colon Y(n)\to Y(n-1)$, defined by $\widetilde{\retr}(B) = f(\retr(f^{-1}(A)))$, is $\eta(L)$-Lipschitz. 
It is easy to see that $\widetilde{\retr}$ is a retraction onto $Y(n-1)$.
\end{proof}

For example, Proposition~\ref{prop:qhmaps} shows that every quasihomogeneous image of Euclidean space $\mathbb R^d$ has the LCP. Freeman~\cite{Freeman11} proved that an unbounded Jordan curve $\Gamma$ in the plane is a quasihomogeneous image of $\mathbb R$ if and only if $\Gamma$ is \emph{bi-Lipschitz homogeneous}, meaning that there exists $L$ such that for any two points $x,y\in \Gamma$ there exists an $L$-bi-Lipschitz self homeomorphism of $\Gamma$ sending $x$ to $y$.  Quasihomogeneous images of $\mathbb R^2$ have been described in~\cite{Freeman13}.

\section{Questions and remarks}\label{sec:questions} 

\begin{question}\label{q:unif-Lip} Is the converse of Corollary~\ref{cor:unif-disconnect} true? That is, does the existence of a uniformly Lipschitz family of retractions $X(n)\to X(m)$ for all $n>m\ge 1$ imply that $X$ is uniformly disconnected? 
\end{question}

Some evidence in favor of affirmative answer is provided by Corollary~5.2 in~\cite{AkoforKovalev} which states that if $X$ contains a bi-Lipschitz image of a line segment as its Lipschitz retract, then for any family of retractions $\retr_n\colon X(n)\to X(n-1)$ the ratio $n^{-1}\Lip \retr_n$ is bounded below by a positive constant. 

\begin{question} 
Does the existence of a Lipschitz retraction $X(n+1)\to X(n)$ imply the existence of Lipschitz  retraction $X(n)\to X(n-1)$? Here $X$ is a general metric space and $n\ge 2$. 
\end{question} 

\begin{question} If a homeomorphic image of $\mathbb R^d$ has the LCP, is it a space of bounded turning~\cite[p.~120]{Heinonen}? Note that quasihomogeneous images of $\mathbb R^d$ have the LCP by Proposition~\ref{prop:qhmaps} and they are spaces of bounded turning. 
\end{question}

\begin{question} Is there a geometric description of the LCP subsets of $\mathbb R$? Some natural examples that are not covered by  \S\ref{sec:linear-sets} are: Cantor-type sets that are not uniformly disconnected (e.g., those of positive measure), countable sets formed by convergent sequences (such as the set in Proposition~\ref{prop:delete-min}), and the set of irrational numbers $\mathbb R\setminus \mathbb Q$. Example~\ref{ex:small-gap} suggests that Cantor-type sets of positive measure do not have the LCP. 
\end{question}

\bibliography{references.bib} 

\begin{thebibliography}{10}

\bibitem{Ahlfors}
Lars~V. Ahlfors.
\newblock {\em Lectures on quasiconformal mappings}, volume~38 of {\em
  University Lecture Series}.
\newblock American Mathematical Society, Providence, RI, second edition, 2006.
\newblock With supplemental chapters by C. J. Earle, I. Kra, M. Shishikura and
  J. H. Hubbard.

\bibitem{Akofor}
Earnest Akofor.
\newblock On {L}ipschitz retraction of finite subsets of normed spaces.
\newblock {\em Israel J. Math.}, 234(2):777--808, 2019.

\bibitem{AkoforKovalev}
Earnest Akofor and Leonid~V. Kovalev.
\newblock Growth rate of {L}ipschitz constants for retractions between finite
  subset spaces.
\newblock {\em Studia Math.}, 260(3):317--326, 2021.

\bibitem{AMS}
Robert~N. Andersen, M.~M. Marjanovi\'{c}, and Richard~M. Schori.
\newblock Symmetric products and higher-dimensional dunce hats.
\newblock {\em Topology Proc.}, 18:7--17, 1993.

\bibitem{Assouad}
Patrice Assouad.
\newblock Plongements lipschitziens dans {${\bf R}^{n}$}.
\newblock {\em Bull. Soc. Math. France}, 111(4):429--448, 1983.

\bibitem{BacakKovalev}
Miroslav Ba\v{c}\'{a}k and Leonid~V. Kovalev.
\newblock Lipschitz retractions in {H}adamard spaces via gradient flow
  semigroups.
\newblock {\em Canad. Math. Bull.}, 59(4):673--681, 2016.

\bibitem{BorovikovaIbragimov}
Marina Borovikova and Zair Ibragimov.
\newblock The third symmetric product of {$\mathbb{R}$}.
\newblock {\em Comput. Methods Funct. Theory}, 9(1):255--268, 2009.

\bibitem{BorsukUlam}
Karol Borsuk and Stanislaw Ulam.
\newblock On symmetric products of topological spaces.
\newblock {\em Bull. Amer. Math. Soc.}, 37(12):875--882, 1931.

\bibitem{CAOZRQ}
Enrique Casta{\~n}eda-Alvarado, Fernando Orozco-Zitli, and M\'{o}nica~A.
  Reyes-Quiroz.
\newblock Lipschitz retractions on symmetric products of trees.
\newblock {\em Indian J. Pure Appl. Math.}, 2021.

\bibitem{DavidSemmes}
Guy David and Stephen Semmes.
\newblock {\em Fractured fractals and broken dreams}, volume~7 of {\em Oxford
  Lecture Series in Mathematics and its Applications}.
\newblock The Clarendon Press, Oxford University Press, New York, 1997.

\bibitem{DiMarco}
Claudio~A. DiMarco.
\newblock Fractal curves and rugs of prescribed conformal dimension.
\newblock {\em Topology Appl.}, 248:117--127, 2018.

\bibitem{FOUCHAL2013219}
S.~Fouchal, M.~Ahat, S.~{Ben Amor}, I.~Lavallée, and M.~Bui.
\newblock Competitive clustering algorithms based on ultrametric properties.
\newblock {\em Journal of Computational Science}, 4(4):219--231, 2013.

\bibitem{Freeman11}
David~M. Freeman.
\newblock Unbounded bilipschitz homogeneous {J}ordan curves.
\newblock {\em Ann. Acad. Sci. Fenn. Math.}, 36(1):81--99, 2011.

\bibitem{Freeman13}
David~M. Freeman.
\newblock Transitive bi-{L}ipschitz group actions and bi-{L}ipschitz
  parameterizations.
\newblock {\em Indiana Univ. Math. J.}, 62(1):311--331, 2013.

\bibitem{GhamsariHerron}
Manouchehr Ghamsari and David~A. Herron.
\newblock Bi-{L}ipschitz homogeneous {J}ordan curves.
\newblock {\em Trans. Amer. Math. Soc.}, 351(8):3197--3216, 1999.

\bibitem{Heinonen}
Juha Heinonen.
\newblock {\em Lectures on analysis on metric spaces}.
\newblock Universitext. Springer-Verlag, New York, 2001.

\bibitem{HeinonenLip}
Juha Heinonen.
\newblock {\em Lectures on {L}ipschitz analysis}, volume 100 of {\em Report.
  University of Jyv\"{a}skyl\"{a} Department of Mathematics and Statistics}.
\newblock University of Jyv\"{a}skyl\"{a}, Jyv\"{a}skyl\"{a}, 2005.

\bibitem{HerronMayer}
David~A. Herron and Volker Mayer.
\newblock Bi-{L}ipschitz group actions and homogeneous {J}ordan curves.
\newblock {\em Illinois J. Math.}, 43(4):770--792, 1999.

\bibitem{Hohti}
Aarno Hohti.
\newblock On absolute {L}ipschitz neighbourhood retracts, mixers, and
  quasiconvexity.
\newblock {\em Topology Proc.}, 18:89--106, 1993.

\bibitem{Kovalev2015}
Leonid~V. Kovalev.
\newblock Symmetric products of the line: embeddings and retractions.
\newblock {\em Proc. Amer. Math. Soc.}, 143(2):801--809, 2015.

\bibitem{Kovalev2016}
Leonid~V. Kovalev.
\newblock Lipschitz retraction of finite subsets of {H}ilbert spaces.
\newblock {\em Bull. Aust. Math. Soc.}, 93(1):146--151, 2016.

\bibitem{KovalevTyson}
Leonid~V. Kovalev and Jeremy~T. Tyson.
\newblock Hyperbolic and quasisymmetric structure of hyperspaces.
\newblock In {\em In the tradition of {A}hlfors-{B}ers. {IV}}, volume 432 of
  {\em Contemp. Math.}, pages 151--166. Amer. Math. Soc., Providence, RI, 2007.

\bibitem{LangPlaut}
Urs Lang and Conrad Plaut.
\newblock Bilipschitz embeddings of metric spaces into space forms.
\newblock {\em Geom. Dedicata}, 87(1-3):285--307, 2001.

\bibitem{MackayTyson}
John~M. Mackay and Jeremy~T. Tyson.
\newblock {\em Conformal dimension}, volume~54 of {\em University Lecture
  Series}.
\newblock American Mathematical Society, Providence, RI, 2010.

\bibitem{MurtaghDownsContreras}
Fionn Murtagh, Geoff Downs, and Pedro Contreras.
\newblock Hierarchical clustering of massive, high dimensional data sets by
  exploiting ultrametric embedding.
\newblock {\em SIAM J. Sci. Comput.}, 30(2):707--730, 2008.

\bibitem{Tuffley}
Christopher Tuffley.
\newblock Finite subset spaces of {$S^1$}.
\newblock {\em Algebr. Geom. Topol.}, 2:1119--1145, 2002.

\end{thebibliography}
\bibliographystyle{plain} 

\end{document}